\documentclass[epsfig,11pt]{llncs}
\usepackage{amsfonts}
\usepackage{amsmath}
\usepackage{amssymb}
\usepackage{epsfig,latexsym,graphicx}
\usepackage{algorithm2e}
\newcommand{\bbV}{{\mathcal V}}
\newcommand{\bbR}{{\mathbb R}}
\newcommand{\bbW}{{\mathcal W}}

\usepackage{hyperref}
\newtheorem{observation}{Observation}

\title{On the Fermat-Weber Point of a Polygonal Chain and its Generalizations\thanks{This work was done when the author was an undergraduate student at the Indian Statistical Institute, Kolkata, India, with financial support from the Department of Science and Technology (DST), Govt. of. India, under the KVPY fellowship award.}}

\author{Bhaswar B. Bhattacharya}

\institute{Department of Statistics, Stanford University\\
{\tt bhaswar.bhattacharya@gmail.com}} 

\date{}

\begin{document}
\maketitle

\begin{abstract}In this paper, we study the properties of the Fermat-Weber point for a set of fixed points, whose arrangement coincides with the vertices of a regular polygonal chain. A $k$-chain of a regular $n$-gon is the segment
of the boundary of the regular $n$-gon formed by a
set of $k ~(\leq n)$ consecutive vertices of the regular $n$-gon. We show that for every odd positive integer $k$, there exists an integer $N(k)$, such that the Fermat-Weber point of a set of $k$ fixed points lying on the vertices
a $k$-chain of a $n$-gon coincides with a vertex of the chain whenever $n\geq N(k)$.
We also show that $\lceil\pi m(m+1)-\pi^2/4\rceil \leq N(k) \leq \lfloor\pi m(m+1)+1\rfloor$, where $k~(=2m+1)$ is any odd positive integer. We then extend this result to a more general family of point set, and give an $O(hk\log k)$ time algorithm for determining whether a given set of $k$ points, having $h$ points on the convex hull, belongs to such a family.\\

\noindent{\bf Keywords:} Computational geometry, Facility location, Fermat-Weber Problem, Optimization, Polygons.
\end{abstract}

\section{Introduction}
\label{sec:intro} The Fermat-Weber point $\bbW(S)$ of a set $S$ of
$n$ points $\{p_1, p_2, \ldots, p_n\}$ in $\bbR^d$ is the point $p$
which minimizes the sum $\sum_{i=1}^n d(p, p_i)$
\cite{drezner,kupitz,weber} where $d(\alpha, \beta)$ denotes the
Euclidean distance between the two points $\alpha$ and $\beta$. The
origin of this problem is attributed to the great Pierre de Fermat
(1601-1665) who, four centuries ago, asked to find a point which
minimizes the sum of Euclidean distances to three fixed points in
the plane. Around the year 1640, Evangelista Torricelli (1608-1647)
devised a geometrical construction for this problem. He showed that
the point minimizing the sum of distances from the three fixed
points is the point inside the triangle determined by the three
fixed points, at which every side of the triangle subtends an angle
of $2\pi/3$. This, however, is  true only when all the interior
angles of the triangle are at most $2\pi/3$. The so-called
complementary problem where one angle of the triangle can be greater
than $2\pi/3$ first appeared in Courant and Robbins' famous book
{\it What is mathematics?}. The solution of the complementary
problem, which states that the optimum always point coincides with
the obtuse vertex of the triangle, was correctly proved later by
Krarup and Vajda \cite{vadja}. The solution of the Fermat-Weber
problem with weights associated with each of the three points is
also known. The solution for positive weights can be found in the
book by Yaglom \cite{yaglom}. Jalal et al. \cite{jalal} completely
describes the solution when negative weights are also allowed.

The sum of Euclidean distances to four fixed points in the plane is
minimized at the point of intersection of the diagonals, when the
fixed points form a convex quadrangle. Otherwise, the sum is
minimized at the fixed point which coincides with the concave corner
of the quadrangle formed by the four fixed points. The convex case
was first solved by Fagnano \cite{fagano}, but the origin of the
solution for the other case remains unknown. Recently Plastria
\cite{plastria} gave new proofs for both the cases and generalized
these results to general metrics and norms.

Bajaj \cite{Bajaj} showed that even for 5 points, the coordinates of
the Fermat-Weber point may not be representable even if we allow
radicals, and that it is impossible to construct an optimal solution
by means of a ruler and a compass. There are only a few patterns
where the location Fermat-Weber point can be determined exactly. A
point set $S$ is said to form an equiangular configuration if there
exists a point $c \notin S$ and an ordering of the points in $S$
such that each two adjacent points form an angle of $2\pi/n$ with
respect to $c$. The Fermat-Weber point of an equiangular
configuration is the point $c$. Anderegg et al. \cite{sigmaangular}
presented a linear time algorithm to identify whether a given set of
points is in equiangular configuration.

However, it is difficult to exactly determine the Fermat-Weber point of a set of fixed points
unless it has a highly symmetric arrangement. As a matter of fact,
it is hard to find the Fermat-Weber point even if all fixed points lie
on a circle \cite{webercircle}.

Haldane \cite{haldane} proved that the Fermat-Weber point is unique
for any point set in $\bbR^d$ $(d \geq 2)$, unless the points all
lie on a single straight line. However, no algorithm for computing
the exact solution to the Fermat-Weber problem is known. The most
famous of all existing algorithms is the iterative algorithm due to
Weiszfeld \cite{weiszfeld}. Later Vardi and Zhang \cite{vardizhang}
gave a simple modification of this algorithm for solving the
Fermat-Weber location problem in $\bbR^d$ with extensions to more
general cost functions. Bose et al. \cite{probose} derived
$\epsilon$-approximation algorithms for the Fermat-Weber problem in
any fixed dimension, using geometric data structures. 

The relevance of the Fermat-Weber problem in location science was
first envisaged in 1909 by Alfred Weber \cite{weber}, when he
studied the locational optimization of a firm in a region.
Thereafter, the problem of minimizing the sum of distances from a
given set of fixed points is referred to as the Fermat-Weber
problem. Since then the min-sum criteria has been as an optimization
criterion in several facility location problems and extensive
research has been done on them over the years (\cite{drezner},
\cite{wes}). Recently, Burkard et al. \cite{fermatweberinverse}
introduced the inverse Fermat-Weber problem, where a set of $n$
points in the plane with nonnegative weights is given, and the
objective is to change the weights at minimum cost such that a
prespecified point in the plane becomes the Fermat-Weber point.
Carmi et al. \cite{carmi} studied the Fermat-Weber point of planar convex objects.
Their bounds were improved by Abu-Affash and Katz \cite{katz}, and later by Dumitrescu et al. \cite{dumitrescu}.
In a related paper, Dumitrescu et al. \cite{dumitrescustars} studied minimum length stars and Steiner stars 
of planar point sets. Apart from its relevance in facility location, the Fermat-Weber
problem also finds importance in statistics, especially in the
definition of medians \cite{ginigalvani} and quantiles \cite{q_pc}
of multivariate data. 

In this paper, we study the properties of the Fermat-Weber point for
a set of fixed points, whose arrangement coincides with the vertices
of a regular polygonal chain. A $k$-chain of a regular $n$-gon (or a
regular polygonal chain of length $k$) is the segment of the
boundary of the regular $n$-gon formed by a set of $k ~(\leq n)$
consecutive vertices of the regular $n$-gon. A $k$-chain of a
regular $n$-gon will be denoted by $C_n(k)$. $\bbW(C_n(k))$ denotes
the Fermat-Weber point of the set of $k$ fixed points which
coincides with the vertices of the chain $C_n(k)$. A chain is said
to be {\it empty} if it has no vertices, that is, $k=0$. In Section
\ref{wpc} we study some of the properties of $\bbW(C_n(k))$. Observe
that when a vertex of a chain $C_n(k)$ is deleted, we get two
smaller chains of the same regular $n$-gon, one of which may be
empty. Now, if $k$ is an odd integer, there exists a vertex of the
chain $C_n(k)$, which when deleted gives two identical smaller
chains. We call this vertex the {\it root vertex} of the chain. We
show that for every odd positive integer $k$, there exists an
integer $N(k)$ such that $\bbW(C_n(k))$, coincides with the root
vertex of $C_n(k)$, whenever $n\geq N(k)$. This can be thought of as
an extension of Courant and Robbins' complementary problem on
triangles. In Section \ref{algorithm} we prove that $\lceil \pi
m(m+1)-\pi^2/4 \rceil \leq N(k) \leq \lceil \pi m(m+1)+1\rceil$,
where $k=2m+1~(m\geq 1)$, is any odd positive integer. In Section
\ref{sec:extension}, we extend this result to a more general family
of point set. We also present an $O(hk\log k)$ time algorithm for
determining whether a given set of $k$ points, having $h$ points on
the convex hull, belongs to such a family. Finally, in Section
\ref{conclusions} we summarize our work and give some directions for
future work.

\section{Fermat-Weber Point of Polygonal Chains}
\label{wpc}

In this section we prove various properties of the Fermat-Weber point of a set of fixed points lying on the vertices of a polygonal chain.
We denote by $\bbV(C_n(k))$ the set of vertices of the chain $C_n(k)$.
In the following we shall assume that the vertices of the chain $C_n(k)$ lie on the circumference
of a unit circle with center at the point $o$, because the Fermat-Weber point of a set of fixed points
remains invariant under uniform scaling.

\begin{figure*}[h]
\centering
\begin{minipage}[c]{0.5\textwidth}
\centering
\includegraphics[width=2.25in]
    {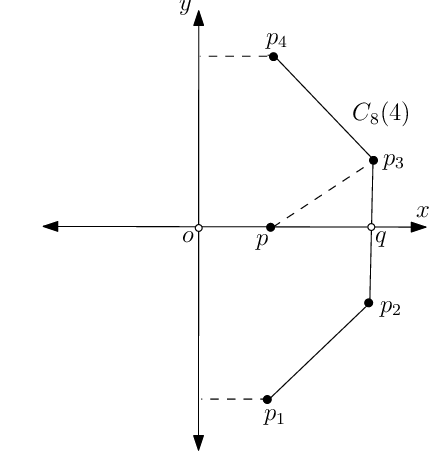}\\
{\small (a)}\\
\end{minipage}%
\begin{minipage}[c]{0.5\textwidth}
\centering
\includegraphics[width=2.25in]
    {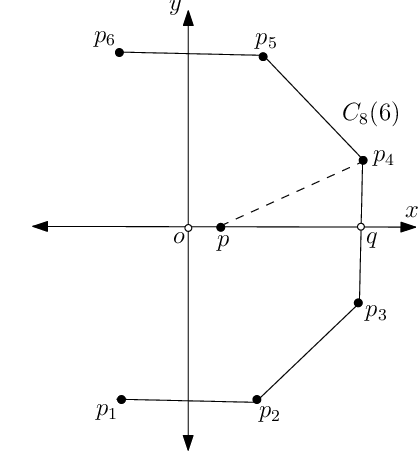}\\
{\small (b)}\\
\end{minipage}
\caption{Even Polygonal Chains: (a) $C_8(4)$: $4$-chain of a regular 8-gon, (b) $C_8(6)$: $6$-chain of a regular 8-gon.}
\label{fig1}
\end{figure*}

Now, depending on whether $k$ is even or odd we have the following two cases:

\begin{description}
\item[{\it Case} 1:]$k =2m$, $m\geq 1$. Let $\bbV(C_n(k))=\{p_1, p_2, \ldots, p_k\}$ be the vertices of
the chain taken in the counter-clockwise direction starting with the
lowermost vertex of the chain, as shown in Figure \ref{fig1}.
Observe that the chain is symmetric about the line $oq$, where $o$
is the circumcenter of $\bbV(C_n(k))$ and $q$  is the midpoint of
the line segment $p_mp_{m+1}$. We call the line $oq$ the {\it line
of symmetry} of the chain. The uniqueness of the Fermat-Weber point
now implies that $\bbW(C_n(k))$ must lie on this line. Consider the
rectangular coordinate system with origin at the point $o$ and $oq$
as the horizontal axis. Let $p:=(x, 0)$, $x > 0$, be a point on the
horizontal axis. The sum of distances from the point $p$ to the
vertices of $C_n(k)$ is then given by
\begin{equation}
\psi(k,n,x)= \sum_{p_i\in\bbV(C_n(k))}d(x, p_i)=\sum_{i=1}^{m}2\sqrt{x^2-2\mu_i(n) x+1}.
\label{eq:case1}
\end{equation}
where $\mu_i(n)=\cos((2i-1)\pi/n)$. The point $p_0$ on the line
segment $\overline{oq}$, where $\psi(k, n, x)$ is minimized is the
location of $\bbW(C_n(k))$. We denote by $|\bbW(C_n(k))|$ the
distance of the point $p_0$ from the origin $o$.


\begin{figure*}[h]
\centering
\begin{minipage}[c]{0.5\textwidth}
\centering
\includegraphics[width=2.3in]
    {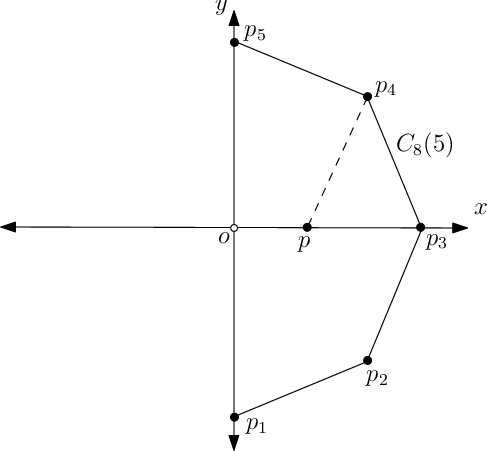}\\
{\small (a)}\\
\end{minipage}%
\begin{minipage}[c]{0.5\textwidth}
\centering
\includegraphics[width=2.3in]
    {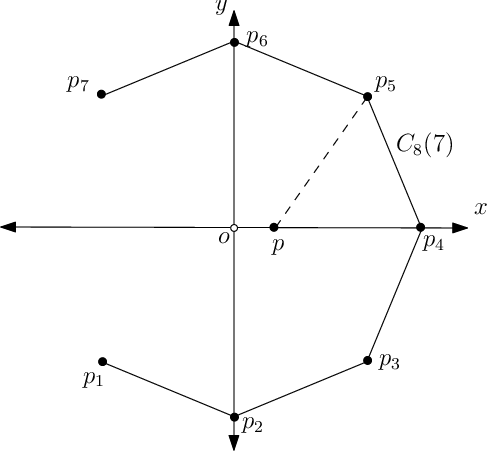}\\
{\small (b)}\\
\end{minipage}
\caption{Odd Polygonal Chains: (a) $C_8(5)$: $5$-chain of a regular 8-gon, (b) $C_8(7)$: $7$-chain of a regular 8-gon.}
\label{fig2}
\end{figure*}

\item[{\it Case} 2:]$k=2m+1$, $m\geq 1$. Let $\bbV(C_n(k))=\{p_1, p_2, \ldots, p_k\}$ be the vertices of
the chain taken in the counter-clockwise direction starting with the
lowermost vertex of the chain, as shown in Figure \ref{fig2}. In
this case, the chain is symmetric about the line $op_{m+1}$, which
then implies that $\bbW(C_n(k))$ must lie on this line. We call the
line $op_{m+1}$ the {\it line of symmetry} of the chain, and the
vertex $p_{m+1}$ the {\it root vertex} of the chain. Consider the
rectangular coordinate system with the circumcenter $o$ of
$\bbV(C_n(k))$ as the origin and the line $op_{m+1}$ as the
horizontal axis. If $p:=(x,0)$, $x\in[0, 1]$, is a point on the
$x$-axis, sum of distances from $p$ to the vertices of the $C_n(k)$
is given by
\begin{equation}
\psi(k, n, x)=  \sum_{p_i\in\bbV(C_n(k))} d(x, p_i)=1-x + \sum_{i=1}^{m}2\sqrt{x^2-2x\lambda_i(n)+1}
\label{eq:case2}
\end{equation}
where $\lambda_i(n)=\cos (2i\pi/n)$. As before, the point $p_0$ on
the line segment $\overline{op_{m+1}}$ where $\psi(k, n, x)$ is
minimized is the location of $\bbW(C_n(k))$. We denote by
$|\bbW(C_n(k)|$ the distance of the point $p_0$ from the origin $o$.

\end{description}

Henceforth, we shall always consider the coordinate system described above while doing any
computation with polygonal chains.

We begin with the following well known observation about the function $\psi(k, n, x)$, which will be repeatedly used in the proofs of
the subsequent results.

\begin{observation}For positive integers $n, k$ such that $n\geq k$ and $k\geq
3$, the objective function  $\psi(k, n, x)$ is strictly convex for
all $x\in[0, 1]$. \label{obs:convex}
\end{observation}
\begin{proof}The double derivative of $\psi(k, n, x)$ for all $x\in [0,1]$ is $$\psi''_x(k, n, x)=\frac{\partial^2}{\partial x^2}\psi(k, n, x)=\sum_{\bbV(C_n(k))}\frac{1-\beta_i^2(n)}{[x^2-2\beta_i(n)x+1]^{3/2}}.$$
where $\beta_i(n)=\lambda_i(n)$ when $k$ is odd and $\beta_i(n)=\mu_i(n)$ when $k$ is even. The result now follows from the fact that $\psi''_x(k, n, x)$ is positive for all $x\in [0, 1]$.
\hfill $\Box$ \end{proof}

Now, we fix a value of $k$ and study the variation in the
Fermat-Weber point of $C_n(k)$ as $n$ varies. Table
\ref{table:4chain} lists the location of the Fermat-Weber points and
the corresponding values of the objective function for $C_n(4)$, for
all $n\leq 18$ varies. Table \ref{table:5chain} lists the same for
the chain $C_n(5)$. These values were obtained by numerically
solving the equation $\frac{\partial}{\partial x}\psi(k, n, x)=0$
using Mathematica 4.0.

\begin{table}
\small
\hspace{0.003\textwidth}
\begin{minipage}{0.46\textwidth}
\small
\centering
\caption{$\bbW(C_n(4))$ for $n\leq 18$}
\begin{tabular}{|c|c|c|}
\hline
$n$ & $\psi(4, n, |\bbW(C_n(4))|)$ & $|\bbW(C_n(4)|$\\
\hline
4 & 4.00000 & 0.000000\\
\hline
5 & 3.80423    &  0.381966\\
\hline
6 & 3.4641  &    0.577352\\
\hline
7 &3.12733   &   0.692021\\
\hline
8 & 2.82843   &   0.765369\\
\hline
9 & 2.57115    &  0.815208\\
\hline
10 & 2.35114    &  0.85065\\
\hline
11 & 2.16256   &   0.876769\\
\hline
12 & 2.00000   &   0.896575\\
\hline
13 & 1.85889   &   0.911956\\
\hline
14 & 1.73553    &  0.924139\\
\hline
15 & 1.62695    &  0.933955\\
\hline
16 & 1.53073  &    0.941979\\
\hline
17 & 1.44497   &   0.948624\\
\hline
18 & 1.36808   &   0.954189\\
\hline
\end{tabular}
\label{table:4chain}
\end{minipage}
\hspace{0.033\textwidth}
\begin{minipage}{0.46\textwidth}
\small
\centering
\caption{$\bbW(C_n(5))$ for $n\leq 19$}
\begin{tabular}{|c|c|c|}
\hline
$n$ & $\psi(5, n, |\bbW(C_n(5))|)$ & $|\bbW(C_n(5)|$ \\
\hline
5 & 5.00000 & 0.000000\\
\hline
6 &  4.83419 & 0.330454\\
\hline
7  & 4.50791  &   0.534378\\
\hline
8 & 4.15356  &    0.667873\\
\hline
9 & 3.81793  &   0.759008\\
\hline
10 & 3.51502  &   0.82332\\
\hline
11 & 3.2466   &  0.869949\\
\hline
12 & 3.01013  &   0.904536\\
\hline
13 & 2.80181    &  0.930659\\
\hline
14 & 2.61783   &  0.950717\\
\hline
15 & 2.45470   &  0.966323\\
\hline
16 & 2.30942   &  0.978603\\
\hline
17 & 2.17944  &   0.98836\\
\hline
18 & 2.06261   &  0.996175\\
\hline
19 & 1.95718   &  1.00000\\
\hline
\end{tabular}
\label{table:5chain}
\end{minipage}
\end{table}

From the values listed these two tables it can be observed that the
distance of $\mathcal W(C_n(k))$ from the circumcenter $o$ increases
as $n$ increases. Hence, we have the following observation:

\begin{observation}If $n_1, n_2, k$ are positive integers, such that $k\leq n_1< n_2$, then
\begin{description}
\item[{\it (i)}]$\psi(k, n_1, |\bbW(C_{n_1}(k)|)> \psi(k, n_2, |\bbW(C_{n_2}(k)|)$,
\item[{\it (ii)}]$|\bbW(C_{n_1}(k)|<|\bbW(C_{n_2}(k))|$, whenever $|\bbW(C_{n_1}(k)|< 1$ and $|\bbW(C_{n_2}(k)|<1$.
\end{description}
\label{ob:ob1}
\end{observation}
\begin{proof}Since $\cos\theta$ is a decreasing function for $\theta\in [0, \pi]$, it follows that if
$n_1 < n_2$, then $\mu_i(n_1) < \mu_i(n_2)$ and $\lambda_i(n_1) < \lambda_i(n_2)$, for every fixed $i$. Equations (\ref{eq:case1}) and (\ref{eq:case2}) now imply that if $n_1< n_2$, then $\psi(k, n_1, x)> \psi(k, n_2, x)$ for all
$x\in [0, 1]$. In particular,
$$\psi(k, n_1, |\bbW(C_{n_1}(k)|)\geq \psi(k, n_2, |\bbW(C_{n_1}(k)|)> \psi(k, n_2, |\bbW(C_{n_2}(k)|),$$
which proves the first part.

Next, let $\psi'(k, n, x)=\frac{\partial}{\partial x}\psi(k, n, x)$.
Now, since $|\bbW(C_{n_1}(k)|< 1$ and $|\bbW(C_{n_2}(k)|<1$, the
minimum of the function $\psi$ lies in the interval $[0, 1)$. This
implies that $\psi'(k, n_1, |\bbW(C_{n_1}(k))|)=\psi'(k, n_2,
|\bbW(C_{n_2}(k))|)=0$. It is easy to see that for any fixed values
of $k$ and $x$, $\psi'(k, n, x)$ is a decreasing function of of $n$.
Therefore, $\psi'(k, n_2, |\bbW(C_{n_1}(k))|)<\psi'(k, n_1,
|\bbW(C_{n_1}(k))|)=0$, and $\psi'(k, n_2, |\bbW(C_{n_2}(k))|)=0$.
Now, since $\psi(k, n, x)$ is convex for all $x\in [0, 1]$, the
derivative $\psi'(k, n, x)$ must be non-decreasing in $x$ in the
interval $[0, 1]$. This implies that
$|\bbW(C_{n_1}(k))|<|\bbW(C_{n_2}(k))|$.
\hfill $\Box$ \end{proof}

From the solution of the complementary problem \cite{vadja} we
already know that the Fermat-Weber point of a 3-chain coincides with
the root vertex of $C_n(3)$ if $n\geq 6$. This fact and the values
in Table \ref{table:5chain}  motivate the formulation of the
following lemma.

\begin{lemma}
For every odd positive integer $k=2m+1~(m\geq 1)$, $\bbW(C_n(k))$
coincides with the root vertex of the chain $C_n(k)$, if and only if
$2\sum_{i=1}^m\sin (i\pi/n)-1\leq 0$. \label{lemma}
\end{lemma}
\begin{proof}($\Rightarrow$) If $\bbW(C_n(k))$ coincides with the root vertex of the chain $C_n(k)$, the
objective function $\psi(k, n, x)$ is minimized at the point $x=1$, for all $x\in [0, 1]$.
This implies that $\psi(k, n, x)$ must be non-increasing on the interval $[0, 1]$, because
$\psi(k, n, x)$ is convex on the interval $[0, 1]$. Therefore,
$$\psi'(k, n, x)=\frac{\partial}{\partial x}\psi(k, n, x)
=2\sum_{i=1}^m\frac{x-\lambda_i(n)}{\sqrt{x^2-2x\lambda_i(n)+1}}-1\leq 0 ~~\forall ~x\in [0, 1].$$
where $\lambda_i(n)=\cos 2(i\pi/n)$. This implies that $\psi'(k, n, 1)=2\sum_{i=1}^m\sin (i\pi/n)-1\leq 0$.

($\Leftarrow$) We know that $\psi'(k, n, 1)=2\sum_{i=1}^m\sin
(i\pi/n)-1\leq 0$. This implies that $\psi'(k, n, x)\leq 0$ for all
$x\in [0, 1]$, because $\psi'(k, n, x)$ is non-decreasing in $x$ on
the interval $[0, 1]$, by the convexity of $\psi$ on $[0, 1]$.
Therefore, $\psi(k, n, x)$ is non-increasing on the interval $[0,
1]$ and is minimized at the endpoint $x=1$, that is, $\bbW(C_n(k))$
coincides with the root vertex of the chain $C_n(k)$.
\hfill $\Box$ \end{proof}

Let us denote $f(n)=2\sum_{i=1}^m\sin(i\pi/n)-1$. Observe that as $n$ increases, $f(n)$ decreases, and if $n$ is large enough $f(n)$ becomes non-positive, since $\lim_{n\rightarrow \infty}f(n)=-1$. Therefore, the condition $f(n)\leq 0$ holds whenever $n$ is sufficiently large, and the following theorem is
immediate.

\begin{theorem}
For every odd positive integer $k=2m+1(m\geq 1)$, there exists a
smallest integer $N(k)$ such that for all $n\geq N(k)$,
$\bbW(C_n(k))$ coincides with the root vertex of $C_n(k)$. Moreover,
$N(k)=\min\{t \in \mathbb N: 2\sum_{i=1}^m\sin (i\pi/t)-1\leq 0\}$.
\hfill $\Box$ \label{theorem}
\end{theorem}

Theorem \ref{theorem} can be viewed as an extension of the complementary problem of Courant and Robbins. It follows from the result of the complementary problem that $N(3)=6$, since the interior angles of a regular hexagon are $2\pi/3$. Theorem \ref{theorem} asserts that $N(k)$ can be determined by checking the sign of $2\sum_{i=1}^m\sin (i\pi/n)-1$ at all the integers till the first time it becomes non-positive. The values $N(k)$ for some small values of $k$ are shown in Table \ref{table:oddtor}.\\



\small
\begin{table*}[h]
\small
\centering
\caption{$N(k)$ for some small values of $k$}
\begin{tabular}{|c|c|c|c|c|c|c|c|c|c|c|c|c|c|}
\hline
\small ~$k$~&~3~&~5~&~7~&~9~&~11~&~13~&~15~&~17~&~19~&~21~&~23~&~25\\
\hline
\small ~$N(k)$~&~6~&~19~& 38&63 & 94& 132&176 & 226&283&346&415&490\\
\hline
\end{tabular}
\label{table:oddtor}
\end{table*}

\normalsize \noindent{\it Note on Even Chains} : Consider a
$k$-chain $C_n(k)$, where $k=2m~(m\geq 1)$ is an even integer. Let
$q$ be the point where the line of symmetry of $C_n(k)$ intersects
the boundary of the chain. As in the case of odd chains, one might
conjecture that $\bbW(C_n(k))$ coincides with $q$ if $n$ is
sufficiently large. This, however, is not true. The coordinates of
the point $q$ are $(\mu_1(n), 0)$, where $\mu_1(n)=\cos(\pi/n)$. It
now follows from Equation (\ref{eq:case1}) that
\begin{equation}
\psi'(k, n, x)=\frac{\partial}{\partial x}\psi(k, n, x)
=2\sum_{i=1}^m\frac{x-\mu_i(n)}{\sqrt{x^2-2x\mu_i(n)+1}} ~~\forall ~x\in [0, \mu_1(n)].
\label{eq:remark_even}
\end{equation}
Note that for every fixed $n$, $\mu_i(n)\leq \pi$, for all $i \in \{1, 2, \ldots, m\}$. Since, $\cos\theta$ is a decreasing function for $\theta\in[0, \pi]$, we have $\mu_1(n)> \mu_i(n)$, for $i\in\{2, \ldots, m\}$. Equation (\ref{eq:remark_even}) now implies that $\psi'(k, n, \mu_1(n))>0$, for every fixed $n$. From the convexity of the function $\psi$ proved in Observation \ref{obs:convex}, we now conclude that $|\bbW(C_n(k))|<\mu_1(n)$, for every fixed $n$.  Therefore, as $n$ increases the Fermat-Weber point of the even chain $C_n(k)$ gradually approaches the point $q$, but it never actually coincides with $q$ for any finite value of $n$. This implies that $N(2m)=\infty$, for $m\geq 2$, and illustrates the impossibility of a result analogous to Theorem \ref{theorem} for even $k$-chains.

It is in fact the root vertex, which dominates the location of $\bbW(C_n(k))$, when $k$ is odd, by
pulling it towards itself as the value of $n$ increases.

\section{Determination of $N(k)$}
\label{algorithm}

In this section we determine bounds on the number $N(k)$ and propose
an algorithm for determining it.

At first, we have the following observation:

\begin{observation}For every odd positive integer $k~(=2m+1)\geq 3$, we have $N(k)\geq m(m+1)$.
\label{obs:N(k)_2k}
\end{observation}

\begin{proof}Observe that $f(m(m+1))=2\sum_{i=1}^{m}\sin (i\pi/(m(m+1)))-1$. Now, using the fact that for $\theta\in [0, \pi]$, $\sin \theta\geq \theta-\theta^3/6$, we get
\begin{eqnarray}
f(m(m+1))&=& 2\sum_{i=1}^{m}\sin \frac{i\pi}{m(m+1)}-1\nonumber\\
 &\geq & 2\sum_{i=1}^{m}\frac{i\pi}{m(m+1)}-2\sum_{i=1}^{m}\frac{i^3\pi^3}{6m^3(m+1)^3}-1\nonumber\\
 &\geq & \pi-\frac{\pi^3}{12 m(m+1)}-1
 \label{eqn:observation_N(k)_lower_bound}
\end{eqnarray}
where the last equation follows from the fact that $\sum_{i=1}^m
i=m(m+1)/2$ and $\sum_{i=1}^m i^3=m^2(m+1)^2/4$. Now, since for all
$m\geq 1$, $m(m+1)\geq 2$ we get from Equation
(\ref{eqn:observation_N(k)_lower_bound}), $f(m(m+1))\geq
\pi-\frac{\pi^3}{24}-1 >0.$ This prove that $N(k)\geq m(m+1)$ for
$k=2m+1$, $m\geq 1$.
\hfill $\Box$ \end{proof}

Using this observation, we now prove the following bounds on $N(k)$.

\begin{theorem}
For every odd positive integer $k~(=2m+1)$, $m\geq 1$, we have $$\lceil \pi m(m+1)-\pi^2/4 \rceil \leq N(k) \leq \lfloor \pi m(m+1)+1\rfloor.$$
\label{th:th2}
\end{theorem}
\begin{proof}Let $C_n(k)$ be a regular polygonal chain of length $k$, where $k=2m+1, m\geq 1$.
From Theorem \ref{theorem} we have $N(k)=\min\{t \in \mathbb N:
2\sum_{i=1}^m\sin (i\pi/t)-1\leq 0\}$. Since $N(k)$ is the smallest
integer of this set, we must have,
$2\sum_{i=1}^m\sin(i\pi/(N(k)-1))-1> 0$. Observe that for every
$i\in\{1, 2, \ldots, m\}$, the function $\sin(i\pi/(N(k)-1))$ is
continuous and differentiable in the interval $(0, i\pi/(N(k)-1))$.
Hence, by the Mean Value Theorem, for every $i\in\{1, 2, \ldots,
m\}$ there exists $\theta^*_i\in(0, i\pi/(N(k)-1))$ such that
$$2\sum_{i=1}^m\sin\frac{i\pi}{N(k)-1}-1=2\sum_{i=1}^m\frac{i\pi}{N(k)-1}\cos\theta_i^*-1>0.$$
This implies, $N(k)<2\pi\sum_{i=1}^m i\cos\theta_i^*+1\leq 2\pi\sum_{i=1}^m i+1=\pi m(m+1)+1$. This proves that for
$m\geq 1$, we have $N(k) \leq \lfloor \pi m(m+1)+1 \rfloor$.

To prove the lower bound, observe that $2\sum_{i=1}^m\sin i\pi/N(k)-1\leq 0$. By the
Mean Value Theorem, for every $i\in\{1, 2, \ldots, m\}$ there
exists $\theta_i\in(0, i\pi/N(k))$ such that
$$2\sum_{i=1}^m\sin\frac{i\pi}{N(k)}-1=2\sum_{i=1}^m \frac{i\pi}{N(k)}\cos\theta_i-1\leq 0.$$
This implies that $N(k)\geq 2\pi\sum_{i=1}^m i\cos\theta_i$. Using the inequality $\cos\theta_i\geq 1-\theta_i^2/2$, for all $i\in\{1, 2, \ldots, m\}$, and the fact that $\theta_i\in(0, i\pi/N(k))$, we get $\cos\theta_i\geq 1-\frac{i^2\pi^2}{2(N(k))^2}$. Now, since from Observation \ref{obs:N(k)_2k}, we know that $N(k)\geq m(m+1)$, we get
$\cos\theta_i\geq 1-\frac{i^2\pi^2}{2m^2(m+1)^2}$, for all $i\in\{1, 2, \ldots, m\}$. Therefore,
\begin{eqnarray}
N(k)\geq 2\pi\sum_{i=1}^m i\cos\theta_i & \geq & 2\pi\sum_{i=1}^m i\left(1-\frac{i^2\pi^2}{2m^2(m+1)^2}\right)\nonumber\\
&= & \pi m(m+1)- \frac{\pi^2}{4}.
\label{eqn:theorem_N(k)_lower_bound}
\end{eqnarray}
This proves that for $m\geq 1$, $N(k) \geq \lceil \pi m(m+1)-\pi^2/4 \rceil$.

\hfill $\Box$ \end{proof}

Observe that using standard trigonometric formulae \cite{loney}, we get $f(n)=2\sum_{i=1}^m\sin(i\pi/n)-1=\frac{2\sin((m-1)\pi/(2n))\sin (m\pi/(2n))}{\sin(\pi/(2n))}-1$. Theorem \ref{th:th2} now immediately implies that the value of $N(k)$ can be computed by checking the sign of the function $f(n)$ at
maximum $\lceil \pi^2/4+1\rceil$ values of $n$. The minimum of value of $n$ at which the function $f$ is non-positive is the value of $N(k)$. Thus, we have the following theorem:

\begin{corollary}For every odd positive integer $k=2m+1~(m\geq 1)$, $N(k)$ can be computed in constant time. \hfill $\Box$
\label{corollary2}
\end{corollary}


\section{Extensions of Theorem \ref{theorem}}
\label{sec:extension}

In this section we show that Theorem \ref{theorem} can be extended
to a larger family of point sets, which have one axis of symmetry. A
set $S$ of $k=2m+1~(m\geq 1)$ points in the plane lying on the
circumference of a unit circle with center at $o$, is said to be
{\it reflection symmetric} if there exits a point $s_0\in S$ such
that, for every point $s_i\in S\backslash\{s_0\}$, in the open
halfplane containing $s_i$, there exists a point $s_i'\in
S\backslash\{s_0, s_j\}$, in the open halfplane not containing
$s_i$, which is the reflection of the point $s_i$ about the line
$os_0$. The point $s_0$ will be called the {\it pivot} of $S$ and
the line $os_0$ the {\it line of symmetry} of $S$. It is easy to see
that a set of points lying on the vertices of a regular polygonal
chain is reflection symmetric.

As any reflection symmetric point set is symmetric about its line of symmetry, the Fermat-Weber point of such a point set must lie on the line of symmetry. For a reflection symmetric point set $S$, with pivot at the point $s_0$, associate the rectangular coordinate system with origin at $o$ and the horizontal axis along the line $os_0$.
In this coordinate system, the coordinate of the point $s_0$ is $(1, 0)$. Denote the points in $S$, as $s_m, \ldots, s_{2}, s_1, s_0, s_1', s_2' \ldots, s_m'$, with the points taken in the clockwise order, such that the point
$s_i'$ is the reflection of the point $s_i$ about the line $os_0$, for each $i\in\{1, 2, \ldots, m\}$ (Figure \ref{fig3}(a)). Let the coordinate of the point $s_i$ be $(\cos\theta_i, \sin\theta_i)$, where $0<\theta_i<\pi$. This implies that the coordinate of its reflection $s_i'$ is $(\cos\theta_i, -\sin\theta_i)$.

Let $p:=(x, 0)$ be any point on the $x$-axis. Now, for $i\in\{1, 2,
\ldots, m\}$, $d(s_i, p)=d(s_i', p)$, and the sum of distances of
the elements in $S$ from a point $p$ is,

\begin{equation}
\psi_S(x)=2\sum_{i=1}^md(s_i, p)+d(s_0, p)=2\sum_{i=1}^m\sqrt{x^2-2x\cos\theta_i +1}+(1-x).
\label{eq:general}
\end{equation}

Let $S$ be a reflection symmetric point set with pivot at the point $s_0$, and $\theta_1, \theta_2, \ldots, \theta_m$ be as described above. Such a set $S$ is said to satisfy {\it Condition A} if $2\sum_{i=1}^m\sin (\theta_i/2)-1\leq 0$. Now, it can be easily seen by following the proof of Lemma \ref{lemma}, that the Fermat-Weber point of $S$ coincides with the pivot $s_0$ if and only if the value of the function $\psi_S'(x)=\frac{d}{dx}\psi_S(x)$ at $x=1$ is negative. Since $\psi_S'(1)=2\sum_{i=1}^m \sin(\theta_i/2)-1$, we have the following theorem,

\begin{theorem}
For every odd positive integer $k=2m+1~(m\geq 1)$, the Fermat-Weber
point of a reflection symmetric point set $S$, with $|S|=k$,
coincides with the pivot of $S$, if and only if $S$ satisfies {\em
Condition A}. \hfill $\Box$ \label{th:rs}
\end{theorem}

\begin{figure*}[h]
\centering
\begin{minipage}[c]{0.5\textwidth}
\centering
\includegraphics[width=2.65in]
    {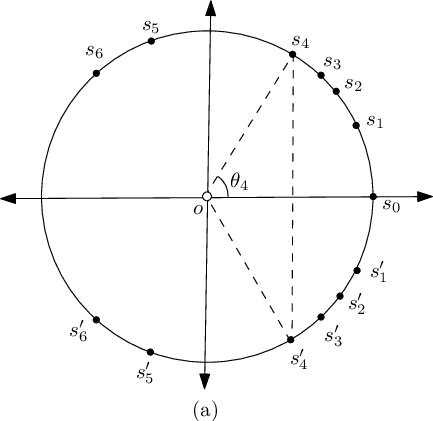}\\
\end{minipage}%
\begin{minipage}[c]{0.5\textwidth}
\centering
\includegraphics[width=3.21in]
    {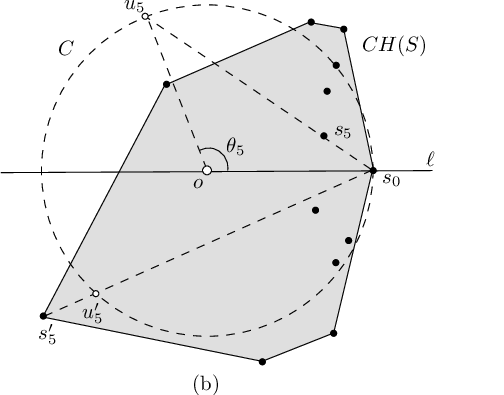}\\
\end{minipage}
\caption{(a) Reflection symmetric point set, (b) Illustration for
the recognition algorithm.} \label{fig3}
\end{figure*}

An interesting property and well-known property of the Fermat-Weber point of a set of non-collinear points is that the position of the Fermat-Weber point remains unchanged if the points in the set are moved along the rays joining them with the Fermat-Weber point of the set \cite{ginigalvani}. This fact can be applied to extend Theorem \ref{th:rs} to an even more general family of point set.

Let $S$ be a reflection symmetric point set with $|S|=k=2m+1~(m\geq 1)$ and $s_0\in S$ be the pivot. Define the {\it extension} of $S$ (to be denoted by $\mathcal A(S)$) as the set of all $k$ element point sets obtained by moving the points of $S\backslash\{s_0\}$ along the rays $\overrightarrow{s_is_0}$, for $s_i\in S\backslash\{s_0\}$.
It is clear that the point $s_0$ belongs to any set $T\in \mathcal A(S)$, and we call it the {\it pivot point} of the family $\mathcal A(S)$. The following result now follows immediately from Theorem \ref{th:rs} and the above discussion.

\begin{theorem}
Let $S$ be a reflection symmetric point set with $|S|=k~(k\geq 3)$
and pivot at the point $s_0$. The Fermat-Weber point of any set
$T\in\mathcal A(S)$ coincides with the pivot point $s_0$ if and only
if the Fermat-Weber point of $S$ coincides with $s_0$, that is, if
and only if $S$ satisfies {\em Condition A}. \hfill $\Box$
\label{th:extension}
\end{theorem}

We now present a simple algorithm which recognizes whether a given
point set belongs to the extension of a reflection symmetric point
set. If it does, the algorithm also determines whether the
Fermat-Weber point of such a set coincides with its pivot, by
verifying {\it Condition A}.

Given a point set $Z$, we denote by $CH(Z)$ the {\it convex hull} of $Z$, and by $|CH(Z)|$ the number of vertices of $Z$ in $CH(Z)$.

\begin{theorem}Given a point set $T$, with $|T|=k=2m+1~(m\geq 1)$ and $|CH(T)|=h$, there exists an $O(hk\log k)$ time algorithm which determines whether $T$ belongs to the extension of some reflection symmetric point set $S$, and whether the Fermat-Weber point of $T$ coincides with the pivot point of $\mathcal A(S)$.
\label{th:algorithm}
\end{theorem}

\begin{proof}Observe that if $T$ belongs to the extension of some set $S$, then the pivot point of the family $\mathcal A(S)$ lies on the convex hull of $S$. Further, a point $s_0$ on the convex hull of $T$ will be the pivot of some reflection symmetric point set if and only if the angle bisectors of the $m$ angles $\angle s_is_0s_i'$, $i=1, 2, \ldots, m$ coincide, where $s_1, s_2, \ldots, s_m, s_m', \ldots, s_2', s_1'$ are the points of $T\backslash\{s_0\}$ ordered by radially sorting them about $s_0$ in the counterclockwise direction.

We begin by finding out the convex hull of $S$ which requires $O(k\log h)$ time \cite{chan,kirkpatrick_seidel}.
For every point in $CH(T)$, we radially sort the remaining points of $T$ about that point, and check whether the angle bisector of the $m$ angles described above coincide. The radial sorting step requires $O(k\log k)$ time after which the $m$ bisectors can be checked in $O(k)$ time. Since this has to be done for all the vertices of the convex hull of $T$, the total running time of the algorithm is $O(hk\log k)$.

Let $S$ be the reflection symmetric point set such that $T\in
\mathcal A(S)$. The set $S$ can now be constructed from $T$ in
$O(k)$ time as follows: The point $s_0$ identified above will
clearly be the pivot point of $\mathcal A(S)$. Let $l$ be the common
angle bisector of the $m$ angles described above (Figure
\ref{fig3}(b)).  Construct the unit circle $C$ with center $o$ on
the line $\ell$ and passing through the the point $s_0$. For every
point $s_i\in T\backslash\{s_0\}$, let $u_i$ be the point where the
ray $\overrightarrow{s_is_0}$ intersects the circumference of $C$
(Figure \ref{fig2}(b)). (Note that if some point $s_i\in T$ lies on
the circumference of $C$, then $u_i=s_i$.) If $S_0=\{u_i|s_i\in
S\backslash\{s_0\}\}$, then $S=S_0\cup\{s_0\}$. {\it Condition A}
for the point set $S$ can now be checked in $O(k)$ time and the
result follows.
\hfill $\Box$ \end{proof}



\section{Conclusions}
\label{conclusions}
In this paper, we have explored the geometric properties of the Fermat-Weber point of polygonal chains.
From the uniqueness of the Fermat-Weber point it is known that the Fermat-Weber point of a regular $n$-gon coincides with
its circumcenter. However, when some vertices of the regular polygon are missing, the Fermat-Weber point can no longer be predicted exactly.
Here, we show that Fermat-Weber point of polygonal chains, which are obtained by deleting a set of consecutive vertices
of a regular polygon, can be predicted exactly in some situations. We show that for every odd positive integer
$k$, there exists a smallest integer $N(k)$ such that for all $n\geq N(k)$, $\bbW(C_n(k))$ coincides with the root vertex of $C_n(k)$.
This interesting geometric result can be thought of as an extension of the complementary problem of Courant and Robbins'. We also extend our results to more a general class point set and give a simple $O(hk\log k)$
time algorithm for identifying whether a given set of $k$ points, with $|CH(S)|=h$, belongs to such a class.

It may be interesting to find generalizations of this result to higher dimensions and to more general distance functions. However, as mentioned by Plastria \cite{plastria}, there does not seem to be much hope in geometrically predicting the Fermat-Weber point unless the given point set has a highly symmetric arrangement.\\

\noindent{\bf Acknowledgement.} The author wishes to thank Prof.
Probal Chaudhuri of Indian Statistical Institute, Kolkata for his
insightful suggestions and for his inspirational guidance. The
constructive comments of the anonymous reviewer are also thankfully
acknowledged. The author also thanks Wolfram Research for inviting the work
to be posted as a live demo in the \href{http://demonstrations.wolfram.com/FermatWeberPointOfAPolygonalChain/}{Wolfram Demonstration Project}.

\bibliographystyle{IEEEbib}
\bibliography{strings,refs,manuals}

\end{document}